\documentclass[a4paper, 10pt,reqno,titlepage]{amsart}
\usepackage[foot]{amsaddr}
\usepackage{tikz}
\usepackage{tkz-euclide}
\usetkzobj{all}
\usepackage{fullpage}
\usepackage{subcaption}
\usepackage [english]{babel}
\usepackage[utf8]{inputenc}
\usepackage{multirow}
\usepackage{amsmath}
\usepackage{amssymb}
\usepackage{amsfonts}
\usepackage{paralist}
\usepackage{braket}
\usepackage{graphicx}
\usepackage{placeins}
\usepackage{wrapfig}
\usepackage{multicol}
\usepackage{nccmath}
\usepackage{longtable}
\usepackage{caption}
\usepackage{hyperref}
\usepackage{amsthm}
\usepackage{cleveref}

\theoremstyle{plain} 
\newtheorem{thm}{Theorem}[section]
\newtheorem{lem}[thm]{Lemma}
\newtheorem{cor}[thm]{Corollary}
\newtheorem{prop}[thm]{Proposition}
\newtheorem{clm}[thm]{Claim}
\theoremstyle{definition}
\newtheorem{defn}{Definition}[section]

\theoremstyle{remark}
\newtheorem{rem}[thm]{Remark}
\newtheorem{ass}[thm]{Assumption}

\newcommand{\C}{\mathbb{C}}
\newcommand{\R}{\mathbb{R}}
\newcommand{\T}{\mathbb{T}}
\newcommand{\N}{\mathbb{N}}
\newcommand{\Z}{\mathbb{Z}}

\newcommand\bbA{{\mathbb{A}}}
\newcommand\bbB{{\mathbb{B}}}
\newcommand\HB{H^{\bullet}}

\newcommand{\sfA}{{\mathsf{A}}}

\newcommand{\dr}{\mathrm{d}}

\newcommand{\de}{\partial}
\newcommand{\ii}{{\mathrm{i}}}

\def\g{\mathfrak{g}}
\def\S{\mathbb{S}}

\def\sfA{\mathsf{A}}
\def\sfB{\mathsf{B}}

\def\sfa{\mathsf{a}}
\def\sfb{\mathsf{b}}

\def\calY{\mathcal{Y}}
\def\calH{\mathcal{H}}
\def\calL{\mathcal{L}}
\def\calV{\mathcal{V}}
\def\calP{\mathcal{P}}
\def\calS{\mathcal{S}}
\def\calD{\mathcal{D}}
\def\calB{\mathcal{B}}
\def\calF{\mathcal{F}}

\def\S{\mathcal{S} }

\def\Seff{S^{\text{eff}}}
\def\calSeff{\mathcal{S}^{\text{eff}}}

\begin{document}
\title{Split Chern-Simons theory in the BV-BFV formalism}
\author{Alberto S. Cattaneo${}^1$}
\author{Konstantin Wernli${}^1$}
\address{${}^1$Institut f\"ur Mathematik, Universit\"at Z\"urich, Winterthurerstrasse 190, 8053 Z\"urich, Switzerland}
\email{alberto.cattaneo@math.uzh.ch, konstantin.wernli@math.uzh.ch}
\author{Pavel Mnev${}^{2,3}$}
\address{${}^2$Max-Planck-Institut f\"ur Mathematik, Vivatsgasse 7,
53111 Bonn, Germany}
 \address{${}^3$St. Petersburg Department of V. A. Steklov Institute of Mathematics of the Russian Academy of Sciences,
Fontanka 27, St. Petersburg, 191023 Russia}
\email{pmnev@pdmi.ras.ru}

\thanks{A. S. C. and K. W. acknowledge partial support of SNF Grants No. 200020-149150/1 and PDFMP2\_137103. This research was (partly) supported by the NCCR SwissMAP, funded by the Swiss National Science Foundation, and by the COST Action MP1405 QSPACE, supported by COST (European Cooperation in Science and Technology). P. M. acknowledges partial support of RFBR Grant No. 13-01-12405-ofi-m.}

\tikzset{middlearrow/.style={
        decoration={markings,
            mark= at position 0.5 with {\arrow{#1}} ,
        },
        postaction={decorate}
    }
}
\begin{abstract}
The goal of this note is to give a brief overview of the BV-BFV formalism developed by the first two authors and Reshetikhin in \cite{Cattaneo2014}, \cite{Cattaneo2015} in order to perform perturbative quantisation of Lagrangian field theories on manifolds with boundary, and present a special case of Chern-Simons theory as a new example. 
\end{abstract}
\maketitle
\tableofcontents
\section{Introduction}
Since the proposal of functorial quantum field theory by Atiyah and Segal (\cite{Atiyah1988},\cite{Segal1988}) mathematical research in this topic has progressed far and in many directions (for a review see e.g. the books \cite{Turaev2010} and \cite{Ohtsuki2002}, or the review article \cite{Schwarz2007}). Recently, the first two authors together with Reshetikhin introduced the BV-BFV formalism, which can be seen either as an extension of functorial QFT to perturbative quantisation or, from another viewpoint, as a method to perturbatively quantise gauge theory in the presence of a boundary. The main idea is to unify the Lagrangian Batalin-Vilkovisky (BV) formalism \cite{Batalin1981,Batalin1983} in the bulk and the Hamiltonian Batalin-Fradkin-Vilkovisky (BFV) formalism \cite{Batalin1983a} on the boundary. \\
One possible application is to shed new light on the relation between perturbative techniques and mathematical ideas that are concepts of non-perturbative quantisation, like the Reshitikhin-Turaev invariants (\cite{Reshetikhin1991}, see also \cite{Kirby1991}), and thus ultimately about non-perturbative results to the path integral itself. In this note, a very first step on this road is taken by applying the formalism to a special form of Chern-Simons theory. \\ 
The note is structured as follows: Section 2 delivers a short overview of the relevant formal concepts via the example of abelian BF theory. Section 3 discusses a variant of Chern-Simons theory known as split Chern-Simons theory, in its BV-BFV formulation. Section 4 computes the state of this theory explicitly in lowest orders on the solid torus, which is a first step of constructing the Chern-Simons invariant for lens spaces. 
\section{Overview of the BV and BV-BFV formalisms}
The goal of this section is to give a very brief introduction to the BV-formalism for on manifolds without boundary, and the BV-BFV formalism on manifolds with boundary, for two special examples. 
\subsection{Perturbative Quantisation of Lagrangian field theories}
Fix a dimension $d$. A Lagrangian field theory assigns to every closed $d$-dimensional manifold a \emph{space of fields} $F_M$ and an \emph{action functional} $S_M \colon F_M \to \R$. This action functional is required to be of the form 
$$S_M[\phi] = \int_M \mathcal{L}[\phi(x),\de\phi(x),\ldots], $$ where $\mathcal{L}$, the so-called \emph{Lagrangian density}, should depend only on the fields $\phi$ and finitely many of their derivatives. The critical points of the action functional are called the classical solutions of the theory, and are obtained by solving the Euler-Lagrange equations, also called equations of motion. \\ 
One way of quantising such a theory, suggested by the path integral from quantum mechanics, is to compute ``integrals'' of the form 
$$\int_{F_M}\mathcal{O}[\phi]e^{\frac{i}{\hbar}S_M[\phi]}\mathcal{D}\phi,$$
where $\mathcal{O}$ is an ``observable'', over the space of fields $F_M$ (these integrals are usually also called path integrals, even though they do not involve any paths). In this note we are only interested in the so-called \emph{vaccuum state} or \emph{partition function} 
\begin{equation}\psi = \int_{F_M}e^{\frac{i}{\hbar}S_M[\phi]}\mathcal{D}\phi\label{PartitionFunction}.\end{equation} 
However, in almost all relevant examples the spaces of fields have infinite dimension, and there is no sensible integration theory at hand. One way to still make sense of such expressions in the limit $\hbar \to 0$ is to use (formally) the principle of stationary phase. This produces an expansion in powers of $\hbar$ around critical points of the action. The terms in such an expansion can conveniently be labelled by diagrams, which after their inventor are called Feynman diagrams. A concise introduction can be found in \cite{Polyak2004}. \\
\begin{rem}[Perturbative expansion]
 We will only consider actions of the form $S = S_0 + S_{\text{int}}$ where $S_0$ is the quadratic part (also called ``free'' or ``kinetic'' part). In this case one usually considers the interaction or perturbation term to be small (``weak coupling'') so we can expand the action around critical points of $S_0$ in powers of the interaction (``coupling constant''), and the integral then can be formally computed from the theory of Gaussian moments, usually referred to as \emph{Wick's theorem} in quantum field theory. 
\end{rem}
\subsection{Perturbative quantisation of gauge theories}\label{BVformalism}
In many cases important for physics and mathematics, the Lagrangian is actually \emph{degenerate}, i.e. its critical points are not isolated, and we cannot apply the stationary phase expansion. This is usually due to the presence of symmetries on the space of fields (known as \emph{local} symmetries in physics - as opposed to \emph{global} symmetries like a translational symmetry on the spacetime) that leave the action invariant. \\
Often this problem can be solved by so-called \emph{gauge-fixing} procedures (a thorough introduction to gauge theories from a physical viewpoint can be found in \cite{Henneaux1994}, a concise introduction to the mathematical formalisms in \cite{Mnev2008}). The common idea is to add more fields, corresponding to the generators of those symmetries, to remove the degeneracies in the Lagrangian. The most powerful gauge-fixing procedure (in the sense that it deals with the most general situation) is the Batalin-Vilkovsky formalism(\cite{Batalin1981,Batalin1983}, for a short introduction to the mathematics see \cite{Fiorenza2008}). We will not discuss it in full generality, but rather explain the idea using the example of abelian BF theory, which will be important later in this note. 
\subsubsection{Abelian BF theory}
Let $M$ be a closed manifold. Abelian BF theory has the space of fields 
$$F_M = \Omega^1(M) \oplus \Omega^{d-2}(M) \ni (A, B).$$
Here $\Omega^p(M)$ denotes the vector space of differential $p$-forms on $M$.
The action functional is
$$S_M[A,B] = \int_M B\wedge \dr A$$
and the critical points are simply closed forms $\dr A=0,\dr B = 0$. Obviously, the critical points are not isolated. In fact, adding any exact form to either $A$ or $B$ will leave the action invariant by Stokes' theorem. Therefore, the symmetries of the theory are generated by $\mathcal{A}:=C^{\infty}(M) \oplus \Omega^{d-3}(M)$. An element $(c,\tau) \in \mathcal{A}$ acts on $F_M$ by $(A,B) \mapsto (A + \dr c, B + \dr \tau)$. Since both the space of fields and the space of symmetries are linear here, the space of symmetries can be identified with the space of generators of the symmetries. We then declare the new space of fields to be 
$$F_M^1 := F_M \oplus \mathcal{A}[1].$$
Here $ \mathcal{A}[1]$ means that we give the fields in $\mathcal{A}$ ghost number 1. 
\begin{rem}[Reducible symmetries]
In this note we will only be concerned with dimension $d = 3$, and we will fix this from now. However, in dimension $D \geq 4$, the symmetries of BF theory are \emph{reducible}, that is, ``the symmetries have some symmetries themselves'': We do not change the symmetry of the action given by $(c,\tau)$ if we add to $\tau$ the differential of a $D-4$-form $\tau_2$. In this case one has to introduce the so-called ``ghosts-for-ghosts'' of ghost number 2, which amounts to adding to the space of fields $\Omega^{D-4}(M)[2]$, and continue all the way until one arrives at $\Omega^{D-D}(M)[D-2]$.
\end{rem}
\begin{rem}[Total degree]
Forms commute or anticommute according to their form degree, i.e. if $\omega$ is a $p$-form and $\tau$ is a $q$-form we have $\omega \wedge \tau = (-1)^{pq}\tau \wedge \omega$. If we introduce ghost fields, fields commute or anticommute according to their \emph{total degree}, which is defined to be the form degree plus the ghost number. In BF theory in 3 dimensions, all fields have total degree 1, so all fields anticommute.  
\end{rem}
These new fields are not enough to make the action nondegenerate. One way to resolve the situation is to pass to the \emph{BV space of fields} 
$$\mathcal{F}_M := T^*[-1]F_M^1 = F^1_M \oplus (F^1_M)^*[-1] = F_M \oplus \mathcal{A}[1] \oplus F_M^*[-1] \oplus \mathcal{A}^*[-2]$$ 
and to restrict the action to a Lagrangian submanifold $\mathcal{L} \subset \calF_M$. By Poincar\'e duality, we identify $F(M)^* = (\Omega^{1}(M)\oplus\Omega^{d-2}(M))^* = \Omega^{d-1}(M)\oplus\Omega^2(M)$ and $\mathcal{A}^* = (\Omega^0(M) \oplus \Omega^{d-3}(M))^*=\Omega^d(M)\oplus\Omega^3(M)$. Denoting the dual fields with a $^+$, we summarise the fields and their degrees: 
\begin{table}[h]

\begin{tabular}{ l || c |c| c }
 Field  & Form degree & Ghost number & total degree = ghost number + form degree \\
 $A$& 1 & 0 & 1 \\
 $B$& d-2 = 1 & 0 & d-2 = 1 \\
 $c$&0&1&1\\
 $\tau$&d-3 = 0&1&d-2=1\\
 $A^+$&d-1=2&-1&1\\
 $B^+$&2&-1&1\\
 $c^+$&d=3&-2&1\\
 $\tau^+$&3&-2&1\\
\end{tabular}
\caption{The fields involved in abelian BF theory in dimension 3, with their form degree, ghost number and total degree} \label{FieldsTable}
\end{table}

The new action is then 
\[ \calS_M = \int_M B \wedge \dr A + A^+\wedge \dr c + B^+\wedge\dr\tau. \]  
\begin{rem}[Superfields]
At this point it very convenient to introduce the ``superfields''
\begin{align*}
\sfA &= c + A + B^+ + \tau^+ \in \Omega^{\bullet}(M), \\
\sfB &= \tau + B + A^+ + c^+ \in \Omega^{\bullet}(M). \\
\end{align*}
The action now simply reads 
$$\calS_M = \int_M \sfB \wedge \dr\sfA, $$ 
where only the integral of the top-degree part is non-zero. 
\end{rem} 
\begin{rem}[Structure of the space of fields] 
The grading by ghost number endows $\mathcal{F}_M$ with the structure of a \emph{graded vector space}. The pairing of fields and anti-fields endows $\mathcal{F}_M$ with a so-called \emph{odd symplectic structure} (odd because it pairs fields whose degrees add up to -1, rather than to 0). If $\delta$ denotes the de Rham differential on $\mathcal{F}$, it is given by 
\begin{equation}
 \omega_M = \int_M \delta\sfA \wedge \delta \sfB. \label{BFBVform}
  \end{equation} As every symplectic structure it induces a Poisson bracket, which in this case is called the \emph{BV bracket}. Also, one has the \emph{BV Laplacian} 
$$ \Delta = \sum_{k=0}^3 (-1)^{k+1} \int_M \frac{\delta^2}{\delta\sfA^{(k)}(x)\delta\sfB^{(k)}(x)},$$ 
where $\sfA^{(k)}$ denotes the $k$-form part of $\sfA$. Together with the BV bracket, it gives $\mathrm{Fun}(\mathcal{F})$ the structure of a so-called \emph{BV algebra}. However, in the infinite-dimensional setting this expression for the BV Laplacian is  very singular, and one has to regularise carefully.
\end{rem}
The BV formalism to compute integral \eqref{PartitionFunction} now proceeds as follows: one picks a Lagrangian subspace $\mathcal{L}$ of $\mathcal{F}_M$ such that the BV action has isolated critical points on $\mathcal{L}$. This is the gauge fixing in the BV formalism. The integral 
\begin{equation} 
\psi = \int_{\mathcal{L}}e^{\frac{i}{\hbar}\calS[\phi]}\mathcal{D}\phi
\end{equation}
can be computed by methods of Feynman diagrams. If the BV action satisfies the \emph{Quantum Master Equation} $\Delta(e^{\frac{i}{\hbar}S}) = 0$, then under deformations of $\mathcal{L}$, the result changes by a $\Delta$-exact term. 
\begin{rem}(Quantum and Classical Master Equations) 
The Quantum Master equation $\Delta(e^{\frac{i}{\hbar}\calS}) = 0$ is equivalent to 
$(\calS,\calS)-2i\hbar\Delta\calS = 0$, where $(\cdot,\cdot)$ is the BV bracket. Expanding $\calS$ as a power series in $\hbar$, the degree 0 part $S_0$ has to satisfy $(S_0,S_0) = 0$. This is called the \emph{Classical Master Equation}. 
\end{rem}
\begin{rem}
The statements above can be made entirely precise and rigorously proven for finite-dimensional spaces of fields. In the infinite-dimensional setting, the BV formalism produces a number of postulates that one has to prove in other ways. 
\end{rem}
\begin{rem} [Perturbative expansion]
Abelian BF theory is an example for a free theory (i.e. $S_{\text{int}} \equiv 0$). For theories that are perturbations of free theories, the gauge-fixing for the free part of the theory can be used to compute the expansion in powers of the coupling constant. We will call theories that are perturbations of abelian BF theory ``BF-like''.  
\end{rem}
\subsubsection{Residual fields} \label{ResidualFields}
It can happen that the degeneracy in the quadratic part of the action does not stem from the gauge symmetries alone. This is the case when the operator in the quadratic part of the action has non-trivial ``zero modes'' i.e. it has zeros that are not related under gauge symmetries. In the case of abelian BF theory, the operator in question is the de Rham differential, while the gauge symmetries are given by shifting the fields by exact forms. It follows that the space of inequivalent zero modes is precisely the de Rham cohomology of $M$.\\ In this case the procedure is as follows. One splits the space of fields $\mathcal{F}_M = \mathcal{Y}' \times \mathcal{Y}''$ into a space of \emph{residual fields}\footnote{Also known as background fields, slow fields, infrared fields.} $\mathcal{Y}'$, consisting of representatives of the zero modes, and a complement $\mathcal{Y}''$ that we will call \emph{fluctuations}\footnote{Otherwise known as fast fields or ultraviolet fields.}. One then only integrates over a Lagrangian subspace of $\mathcal{Y}''$, so that the result depends on the residual fields. This yields the definition of the \emph{effective action}: 
\begin{equation}
e^{\frac{i}{\hbar}\calSeff(\phi')} = \int_{\phi''\in\mathcal{L}\subset \mathcal{Y}''}e^{\frac{i}{\hbar}S(\phi',\phi'')}\mathcal{D}\phi''.
\end{equation}
To be compatible with the BV formalism, $\calY'$ and $\calY''$ should be odd symplectic themselves, such that $\calF_M$ has the product structure. In this case, one can prove that in the finite-dimensional case, the QME for the action on $\mathcal{F}$ induces the QME for the effective action. In the case at hand of abelian BF theory, we choose a finite-dimensional space of residual fields, the de Rham cohomology, and one can prove explicitly that the effective action satisfies the QME. 
In the case of abelian BF theory, $\mathcal{Y}'$ should be given by representatives of the de Rham cohomology of $M$. Such a splitting (and a suitable choice of Lagrangian) can then be found e.g. by Hodge decomposition.

\subsection{On manifolds with boundary}\label{ManifoldsWithBdry}
We will now consider the case of manifolds with boundary. The strategy that is compatible with the mathematical idea of gluing of manifolds along boundary components is not to fix  boundary conditions, but instead to think of the state as a functional on the possible boundary fields. \\ 
Consider first the case of a theory without gauge symmetries. Under some assumptions, one can show that a $d$-dimensional field theory induces a space of fields $F^{\de}_{\Sigma}$ on $(d-1)$-dimensional manifolds $\Sigma$ that has a natural symplectic structure. The space of states should be a quantisation of this symplectic manifold. In many examples, $F^{\de}_{\Sigma}$ is actually an affine space, and one can define a quantisation from a Lagrangian polarisation\footnote{This is basically a choice of coordinates and canonically conjugate momenta, similar to the $p$ and $q$ variables in quantum mechanics.} with a smooth leaf space (examples of this are the position or momentum space). $B_{\Sigma}$. In this case, the space of states is the space of functionals on $B_{\Sigma}$. If $\Sigma = \de M$, there is a surjective submersion $F_M \to F^{\de}_{\de M}$ given by restriction of fields to the boundary. If we denote by $p$ the composition of this map with the projection $F^{\de}_{\de M} \to B_{\de M}$, we can define the state by the ``integral'' 
\begin{equation}
\widehat{\psi}_M(\beta) = \int_{p^{-1}(\beta)}e^{\frac{i}{\hbar}S[\phi]}\mathcal{D}\phi.
\end{equation}
\subsection{The BV-BFV formalism}
Now we want to combine this with the method to deal with gauge theories discussed above. Given a space of BV fields $\mathcal{F}_M$ for every $d$-dimensional manifold $M$, there is again an induced space of fields $\mathcal{F}^{\de}_{\Sigma}$ on $d-1$-dimensional manifolds endowed with what is called a BFV structure (see \cite{Schaetz2008} for a mathematical discussion of BFV structure). The result is what is called a \emph{BV-BFV manifold}, whose definition we will now repeat. 
\begin{defn}[BFV manifold]
A \emph{BFV manifold} is a triple $(\mathcal{F},\omega,Q)$, where 
\begin{itemize}
\item $\mathcal{F}$ is a $\Z$-graded manifold, 
\item $\omega = \delta\alpha$ is an exact symplectic form on $\mathcal{F}$,
\item $Q$ is a degree $+1$ vector field,
\end{itemize}
such that
\begin{itemize}
\item $Q$ is symplectic for $\omega$, i.e. $L_Q\omega = 0$, 
\item $Q$ is cohomological, i.e. $Q^2 = 0$ or equivalently $[Q,Q] = 0$. 
\end{itemize}
\end{defn}
For degree reasons this implies the existence of a degree 0 Hamiltonian function $S$ for $Q$, i.e. $\iota_Q\omega = \delta S$ (and the datum of such function specifies a cohomological symplectic vector field) and this function $S$ automatically satisfies the Classical Master Equation $(S,S) = 2\iota_Q\iota_Q\omega = 0$. The $\Z$-grading of the manifold is the ghost number we briefly explained above. 
\begin{defn}[BV-BFV manifold] 
A \emph{BV-BFV manifold} over a given BFV manifold $(\mathcal{F}^{\de},\omega^{\de}=\delta\alpha^{\de},Q^{\de}$ is a quintuple $(\mathcal{F},\omega,Q,S,\pi)$ where 
\begin{itemize}
\item $\mathcal{F}$ is a $\Z$-graded manifold, 
\item $\omega$ is a degree -1 symplectic form, 
\item $Q$ is a degree +1 cohomological  vector field,
\item $\calS$ is a degree 0 function on $\mathcal{F}$,
\item $\pi$ is a surjective submersion $\mathcal{F} \to \mathcal{F}^{\de}$,
\end{itemize}
such that 
\begin{itemize}
\item $\delta\pi(Q) = Q^{\de}$,
\item $\iota_Q\omega = \delta S + \pi^*\alpha^{\de}$.
\end{itemize}
\end{defn} 
The axioms imply the \emph{modified Classical Master Equation} 
\begin{equation}
\frac{1}{2}\iota_{Q}\iota_{Q}\omega_M - \pi^*\calS^{\de} = 0. \label{mCME}
\end{equation}
In many cases, the BV structure on the bulk and the BFV structure on the boundary look very similar in the superfield formalism. \\
Let us look at the example of abelian BF theory on a 3-manifold $M$ with boundary $\de M$ that is included via $\iota \colon \de M \to M$. Let $\mathcal{F}_M$ be the space of BV fields $\Omega^{\bullet}(M)[1] \oplus \Omega^{\bullet}(M)[1] \ni (\sfA,\sfB)$. Denote by $\sfA^{\de}:=\iota^*\sfA,\sfB^{\de}:=\iota^*\sfB$ the restrictions of these fields to the boundary. Then the space of boundary BFV fields is $\mathcal{F}^{\de}_{\de M}=\Omega^{\bullet}(\de M)[1] \oplus \Omega^{\bullet}(\de M)[1] \ni (\sfA^{\de},\sfB^{\de})$. The symplectic form and action have the same form as before
\begin{align*}\omega^{\de}_{\de M}  &= \int_{\de M} \delta \sfA^{\de} \wedge \delta \sfB^{\de}, \\
S^{\de}_{\de M} &= \int_{\de M} \sfB^{\de} \wedge \dr\sfA^{\de},
\end{align*}
and the corresponding Hamiltonian vector field on $\mathcal{F}^{\de}_{\de M}$  is 
\[ Q^{\de}_{\de M} = \int_{\de M} \dr\sfA^{\de} \frac{\delta}{\delta\sfA^{\de}} + \dr\sfB^{\de} \frac{\delta}{\delta\sfB^{\de}}.\]
However, considering table \ref{FieldsTable} and that the dimension of $\de M$ is 2, notice that $\omega^{\de}_{\de M}$ pairs fields of \emph{opposite} ghost number, and thus has degree 0. I.e., $(\mathcal{F}^{\de}_{\de M},\omega^{\de}_{\de M},Q^{\de}_{\de M})$ is a BFV manifold. 
\begin{clm} \label{BFBVBFV}
If we denote 
$$Q_M = \int_M \dr\sfA\frac{\delta}{\delta\sfA} + \dr\sfB\frac{\delta}{\delta\sfB}$$
and $\pi_M = \iota^* \colon \mathcal{F}_M \to \mathcal{F}^{\de}_{\de M}$ the restriction of fields to the boundary, then in abelian BF theory the quintuple
$(\mathcal{F}_M,\omega_M,Q_M,S_M,\pi_M)$ is a BV-BFV manifold over the BFV manifold $(\mathcal{F}^{\de}_{\de M},\omega^{\de}_{\de M},Q^{\de}_{\de M})$. 
\end{clm}
\begin{proof}
We will just prove the central BV-BFV identity $\iota_{Q_M}\omega_M = \delta S_M + \pi^*\alpha^{\de}_{\de M}$. Notice that the de Rham differential on $\mathcal{F}_M$ is given by 
$$ \delta = \int_M \delta \sfA \frac{\delta}{\delta\sfA} + \delta\sfB\frac{\delta }{\delta\sfB}$$ and one choice of $\alpha^{\de}_{\de M} $ is 
$$ \alpha^{\de}_{\de M} = \int_{\de M} \sfB \wedge \delta \sfA. $$On the one hand,
$$\iota_{Q_M}\omega_M = \int_M \dr\sfA \wedge \delta\sfB + \delta\sfA \wedge\dr\sfB.$$  
On the other hand, integrating by parts yields
\begin{align*}\delta \calS_M &= \delta \int_M \sfB \wedge \dr\sfA = \int_M \sfB \wedge \dr \delta A + \int_M \delta B \wedge \dr\sfA \\
&= \int_M \dr\sfB \wedge \delta \sfA + \int_M \delta B \wedge \dr\sfA- \int_{\de M} \sfB \wedge \delta\sfA = \iota_{Q_M}\omega_M - \pi^*_M\alpha^{\de}_{\de M}.
\end{align*}
\end{proof}
\subsection{The quantum BV-BFV formalism}
We now explain the data of a quantum BV-BFV theory and show to quantise in the example of abelian BF theory, before turning to the example of Chern-Simons theory. The perturbative quantisation of a BV-BFV theory consists of the following data: 
\begin{enumerate}
\item A cochain complex $(\mathcal{H}^{\mathcal{P}}_{\Sigma},\Omega^{\mathcal{P}}_{\Sigma})$ for every $(d-1)$-manifold $\Sigma$ with a choice of polarisation in $\mathcal{F}^{\de}_{\Sigma}$. 
\item A finite-dimensional BV manifold $(\mathcal{V}_M,\Delta_{\mathcal{V}_M})$ - called the space of \emph{residual fields} - associated to every $d$-manifold $M$ and polarisation $\mathcal{P}$ on $\mathcal{F}^{\de}_{\de M}$. 
\item Let $\widehat{\mathcal{H}}^{\mathcal{P}}_M := \mathcal{H}^{\mathcal{P}}_{\de M} \hat{\otimes} C^{\infty}(\mathcal{V}_M)$ and endow it with the two commuting coboundary operators $\widehat{\Omega}^{\mathcal{P}}_{M} := \Omega^{\mathcal{P}}_{\de M } \otimes \operatorname{id}$ and $\widehat{\Delta}^{\mathcal{P}}_M = \operatorname{id} \otimes \Delta_{\mathcal{V}_M}$. Then we require the existence of a \emph{state} $\widehat{\psi}_M$ satisfying the \emph{modified Quantum Master Equation} (mQME) 
\begin{equation}
(\hbar^2\widehat{\Delta}^{\mathcal{P}}_M + \widehat{\Omega}^{\mathcal{P}}_{M})\widehat{\psi}_M = 0,
\end{equation}
the quantum counterpart of the mCME \eqref{mCME}.
\end{enumerate}
Some comments are in order. The cochain complex $(\mathcal{H}^{\mathcal{P}}_{\Sigma},\Omega^{\mathcal{P}}_{\Sigma})$ is to be constructed as a geometric quantisation of the symplectic manifold $\mathcal{F}^{\de}_{\de M}$ with the polarisation $\calP$ and the action $\calS^{\de}_{\de M}$. Sometimes the boundary action needs to be adapted to the polarisation\footnote{To enforce vanishing of $\alpha^{\de}_{\de M}$ along fibers of $\calP$.}, this can be done by adding a term depending only on boundary fields to it. The new action is denoted $\calS^{\calP}_{\de M}$. The general construction of the boundary quantisation is not important in this note. More important is the idea of residual fields that was explained in section \ref{ResidualFields}. The state is then computed by combining the methods of sections \ref{BVformalism} and \ref{ManifoldsWithBdry}. Again, assume we have a polarisation $\mathcal{P}$ of $\mathcal{F}^{\de}_{\de M}$ with smooth leaf space $\mathcal{B}^{\mathcal{P}}_{\de M}$. In this case we can identify $\mathcal{H}^{\mathcal{P}}_{\Sigma} = \operatorname{Fun}(\mathcal{B}^{\mathcal{P}}_{\de M})$.  We will then assume that actually $\mathcal{F}_{M} = \mathcal{B}^{\mathcal{P}}_{\de M} \times \calY$ so that the fibers of the projection $p\colon \mathcal{F}_M \to \mathcal{B}^{\mathcal{P}}_{\de M}$ are just $\{b\} \times \mathcal{Y}$. We then split $\mathcal{Y} = \mathcal{V}_M \times \mathcal{Y}''$ into a space of residual fields and fluctuations $\mathcal{Y}''$.  
Then we can finally define the state $\widehat{\psi}_M$ by 
\begin{equation}
\widehat{\psi}_M(b, \phi) = \int_{\mathcal{L}\subset\mathcal{Y}''}e^{\frac{i}{\hbar}S_M(b,\phi,\phi'')}\mathcal{D}\phi'' \in  \widehat{\mathcal{H}}^{\mathcal{P}}_M = \mathcal{H}^{\mathcal{P}}_{\de M} \hat{\otimes} C^{\infty}(\mathcal{V}_M).
\end{equation}
Again, we define the BV effective action by 
\begin{equation}
\widehat{\psi}_M(b, \phi) = e^{\frac{i}{\hbar}\Seff(b,\phi)}.
\end{equation}
Instead of entering a general discussion of the above, let us continue the example of abelian BF theory. \pagebreak
\subsection{Abelian BF theory in the quantum BV-BFV formalism}
\subsubsection{Polarisations}Here there are two easy polarisations on $\mathcal{F}^{\de}_{\de M} = \Omega^{\bullet}(\de M)[1] \oplus \Omega^{\bullet}(\de M)[1]$, namely the ones given by $\frac{\delta}{\delta\sfA^{\de}}$ (whose leaf space  can be identified with the $\sfB^{\de}$ fields) and $\frac{\delta}{\delta\sfB^{\de}}$ (whose leaf space can be identified with the $\sfA^{\de}$ fields). 
Let now $M$ be a manifold with boundary $\de M = \de_1 M \sqcup \de_2 M$. We then define the polarisation $\mathcal{P}$ to be the $\frac{\delta}{\delta\sfB^{\de}}$-polarisation on $\de_1M$ and the $\frac{\delta}{\delta\sfA^{\de}}$-polarisation on $\de_2M$, so that we have the leaf space $\mathcal{B}^{\mathcal{P}}_{\de M} = \Omega^{\bullet}(\de_1M)[1] \oplus \Omega^{\bullet}(\de_2M)[1]$, we denote the coordinates on it by $(\bbA,\bbB)$. The correct way to adapt the action is to subtract the term $f^{\calP}_{\de M} = \int_{\de M} \sfB^{\de} \wedge \sfA^{\de}$ from it. \\
\subsubsection{Choosing a splitting.} We now split the space of fields $\mathcal{F}_M$ by choosing extensions $\widetilde{\bbA},\widetilde{\bbB}$ and splitting $\sfA = \widetilde{\bbA} + \widehat{\sfA}, \sfB = \widetilde{\bbB} + \widehat{\sfB}$ where $\widehat{\sfA}$ and $\widehat{\sfB}$ restrict to 0 on $\de_1M$ resp. $\de_2M$. It is discussed in \cite{Cattaneo2015} that one needs to require the extensions to be discontinuous extensions by 0 outside of the boundaries. One way to make this more precise is to work with a family of regular decompositions approximating this singular one, resulting a family of states that only in the limit will satisfy the mQME. We will therefore choose these extensions and identify $\widetilde{\bbA} = \bbA, \widetilde{\bbB} = \bbB$. This is our splitting $\mathcal{F}_M = \mathcal{B}^{\mathcal{P}}_{\de M} \times \mathcal{Y}$. \\
\subsubsection{Residual fields and fluctuations.} We now want to split $\mathcal{Y}$ into residual fields and fluctuations. As discussed above, in abelian BF theory the residual fields should contain the de Rham cohomology of $M$. In the case with boundary, for our polarisation, the minimal space of residual fields is 
\begin{equation}
\mathcal{V}_M = \HB(M,\de_1M)[1] \oplus \HB(M,\de_2M)[1].
\end{equation}
We choose representatives $\chi_i \in \Omega^{\bullet}_{closed}(M,\de_1M)$ and $\chi^j \in \Omega^{\bullet}_{closed}(M,\de_2M)$ such that their cohomology classes form a basis of $\HB(M,\de_1M)$ resp. $\HB(M,\de_2M)$ and $\int_M\chi_i\wedge\chi^j = \delta_i^j$. 
Then, we write $\sfa = \sum_iz^i\chi_i,\sfb = \sum_iz_i^+\chi^i$ for elements of $\calV_M \subset \calF_M$. The BV Laplacian $\Delta_{\calV_M}$ is then 
\begin{equation}
\Delta_{\calV_M} = \sum_i - \frac{\de}{\de z^i}\frac{\de}{\de z_i^+}.
\end{equation}
A possible way to choose such a basis, a complement $\mathcal{Y}''$ and a Lagrangian $\mathcal{L} \subset \mathcal{Y}'' $ is to pick a Riemannian metric and use Hodge decomposition on manifolds with boundary. We will avoid the details of this lengthy discussion, referring the interested reader again to \cite{Cattaneo2015}, and simply assume we can decompose the fields $\widehat{\sfA}=\sfa+\alpha, \widehat{\sfB}=\sfb+\beta$  into residual fields and fluctuations.
\begin{rem}[Decomposition of the action] \label{FreeActionSplittingRemark}
The decomposition of the fields also induces a decomposition of the adapted action
\begin{equation} \mathcal{S^{\mathcal{P}}_{\de M}} = \widehat{\mathcal{S}}_{M,0} +  \mathcal{S}_M^{\text{back}}+\mathcal{S}_M^{\text{source}}, \label{FreeActionSplittingEquation}
\end{equation}
where 
\begin{align*}
\widehat{\mathcal{S}}_{M} &= \int_M \beta\wedge\dr\alpha, \\
\mathcal{S}_M^{\text{back}}&= - \left(\int_{\de_2M} \bbB \wedge \sfa  + \int_{\de_1M} \sfb \wedge \bbA\right),\\ 
 \mathcal{S}_M^{\text{source}} &= - \left(\int_{\de_2M} \bbB \wedge \alpha + \int_{\de_1M} \beta \wedge \bbA\right).
\end{align*}
However, since our extension of the boundary fields is singular, this requires some care to obtain. The trick is to decompose the action for non-singular extensions $\widetilde{\bbA},\widetilde{\bbB}$, integrating the $\dr\widetilde{\bbA}$ by parts terms, and then passing to the singular decomposition. 
\end{rem}
\subsubsection{The state.}
We now would like to compute the state 
\begin{equation}\widehat{\psi}_M(\bbA,\bbB,\sfa,\sfb) = \int_{(\alpha,\beta)\in \calL}e^{\frac{i}{\hbar}\calS^{\calP}_M(\bbA + \sfa + \alpha,\bbB + \sfb + \beta)}\calD\alpha\calD\beta \quad \in \widehat{\calH}^{\calP}_M = \mathrm{Fun}(\calB^{\calP}_{\de M}) \hat{\otimes} C^{\infty}(\calV_M).\label{FreeState}\end{equation}
Applying decomposition \eqref{FreeActionSplittingEquation} of the action, and the general theory of performing such Gaussian integrals in quantum field theory, it is enough to understand the integral 
\begin{equation} T_M := \int_{\calL}e^{\frac{i}{\hbar}\widehat{\calS}_M}\calD\alpha\calD\beta \label{TorsionInt}. \end{equation}
For our purposes it is enough to say that $T_M$ is a number independent of the choice of $\calL$ (but that can depend on our choice of representatives of cohomology).   The integral \eqref{FreeState} can then be expressed in terms of the so-called \emph{propagator}\footnote{Also known to physicists as 2-point function or - slightly abusing language - Green's function.} 
\begin{equation}
\eta(x_1,x_2)  = \frac{-1}{T_M}\frac{1}{i\hbar}\int_{\calL}e^{\frac{i}{\hbar}\widehat{\calS}_M}\alpha(x_1)\beta(x_2)\calD\alpha\calD\beta. \label{BFProp} 
\end{equation}
Namely, 
\begin{equation} 
\widehat{\psi}_M(\bbA,\bbB,\sfa,\sfb) = T_Me^{\frac{i}{\hbar}\Seff(\bbA,\bbB,\sfa,\sfb)},
\end{equation}
with 
\begin{equation} 
\Seff(\bbA,\bbB,\sfa,\sfb) = - \left(\int_{\de_2M} \bbB \wedge \sfa  - \int_{\de_1M} \sfb \wedge \bbA\right) - \int_{\de_2M\times\de_1M}\pi^*_1\bbA\wedge\eta\wedge\pi^*_2\bbB.
\end{equation}
\subsubsection{The propagator} The propagator $\eta$ is a $(d-1)$-form on the configuration space $C^0_2(M) = \{ (x_1,x_2) \in M \times M: x_1 \neq x_2 \}$ that vanishes for $x_2 \in \de_1M$ or $x_1 \in \de_2M$. It is determined by our choice of gauge fixing Lagrangian. It has two important properties: 
\begin{itemize}
\item Its differential satisfies 
\begin{equation}
\dr\eta = \sum_i (-1)^{\deg \chi_i} \pi_1^*\chi_i\pi_2^*\chi^i. \label{DifferentialOfPorp}
\end{equation}
\item For any $x\in M$, if we fix a chart $\phi \colon U \to \R^3$ satisfying $\phi(x) = 0$, then 
\begin{equation}
\lim_{\varepsilon \to 0}\int_{y \in\de B_{\varepsilon}(0)}\eta(\phi^{-1}(y),x) = 1 = -\lim_{\varepsilon \to 0}\int_{y \in \de B_{\varepsilon}(0)}\eta(x,\phi^{-1}(y)).\label{IntegralofProp}
\end{equation}

\end{itemize}
A choice of such a propagator (and representatives of cohomology) also leads to the definition of a gauge-fixing Lagrangian. For computations with Feynman diagrams it is often desirable to have a propagator satisfying also 
\begin{itemize}
\item \begin{equation} \int_{y \in M} \eta(x,y)\chi_i(y) = \int_{x \in M} \chi^i(x)\eta(x,y) = 0, \label{Koiequalzero}
\end{equation}
\item \begin{equation} \int_{y \in M} \eta(x,y)\eta(y,z) = \int_{x \in M} \eta(z,x)\eta(x,y) = 0. \label{KoKequalzero}
\end{equation}
\end{itemize} 
These properties are not automatic from the definition but they can always be satisfied by picking a suitable $\mathcal{L}$ (see section 4 in\cite{Cattaneo2008} for a discussion on manifolds without boundary, arguments there can be adapted to the case with boundary using machinery in \cite{Cattaneo2015}).
\subsubsection{mQME} In the case of abelian BF theory, the quantisation of the boundary is simply the ``standard'' or ``canonical'' quantisation. It is obtained by the following recipe: In the boundary action, on $\de_1M$ we have to replace every occurence of $\widehat{\sfB}$ by $(-i\hbar\frac{\delta}{\delta \bbA})$, on $\de_2 M$, $\widehat{\sfA}$ has to replaced by $(-i\hbar\frac{\delta}{\delta\bbB})$. Here we have to integrate by parts to do so. The result is 
\begin{equation}
\Omega^{\calP}_{\de M} = (-i\hbar)\left(\int_{\de_1M} \dr\bbA\frac{\delta}{\delta\bbA} + \int_{\de_2M} \dr\bbB \frac{\delta}{\delta\bbB}\right).
\end{equation}
\begin{clm} The state $\widehat{\psi}_M$ defined by \eqref{FreeState} satisfies the mQME \begin{equation}
(\hbar^2\widehat{\Delta}^{\mathcal{P}}_M + \widehat{\Omega}^{\mathcal{P}}_{M})\widehat{\psi}_M = 0.
\end{equation}
\end{clm}
\begin{proof} Since $ \Seff $ is only linear in coordinates on $\calV_M$, it is immediate that $\Delta\Seff = 0$. In this case $(\hbar^2\Delta + \Omega)e^{\frac{i}{\hbar}\Seff}  = -\frac{1}{2}(\Seff ,\Seff)e^{\frac{\ii}{\hbar}\Seff} + \Omega e^{\frac{\ii}{\hbar}\Seff}. $ Only the first two terms in the action depend on the residual fields and hence contribute to the BV bracket. Also, only the bracket of $\sfb$ with $\sfa$ is nontrivial, so we have 
\begin{align*}\frac{1}{2}(\calSeff,\calSeff) &= \left(\int_{\de_2M} \bbB \wedge \sfa  , \int_{\de_1M} \sfb \wedge \bbA\right) = \sum_{i,j} \left(\int_{\de_2M} \bbB \wedge z^i\chi_i , \int_{\de_1M} z_j^+\chi^j \wedge \bbA\right) \\
&= \sum_i (-1)^{\deg z^i}\int_{\de_2M} \bbB \wedge \chi_i  \int_{\de_1M} \chi^j \wedge \bbA,\end{align*}
since $(z^i,z_j^+) = (-1)^{\deg z^i}\Delta(z^iz_j^+) = (-1)^{\deg z^i}$. On the other hand, 
\begin{align*}\Omega e^{\frac{\ii}{\hbar}\Seff} &= \left(\left(\int_{\de_1M} \dr\bbA\frac{\delta}{\delta\bbA} + \int_{\de_2M} \dr\bbB \frac{\delta}{\delta\bbB}\right)\calSeff\right)e^{\frac{\ii}{\hbar}\calSeff} = \left(\int_{\de_2M\times\de_1M}\pi^*_1\bbA\wedge\dr\eta\wedge\pi^*_2\bbB\right)  e^{\frac{\ii}{\hbar}\calSeff}\\
&=\sum_i (-1)^{\deg \chi^i+1}\int_{\de_2M} \bbB \wedge \chi_i  \int_{\de_1M} \chi^j \wedge \bbA,\end{align*}
where we integrated by parts and used property \eqref{DifferentialOfPorp}. Now the claim follows from the fact that $\deg z^i = 1 - \deg \chi^i$.
\end{proof}

\subsubsection{Dependence of the state on the gauge-fixing.} Clearly, the state defined in \eqref{FreeState} depends on the choice of the gauge-fixing. However, one can show (and, by finite-dimensional arguments, this is supposed to hold in any quantum BV-BFV theory) that, upon deformations of the gauge fixing, the state changes as 
\begin{equation}
\frac{d}{dt}\widehat{\psi} = (\hbar^2\widehat{\Delta}_M + \widehat{\Omega}^{\calP}_M)\widehat{\zeta}
\end{equation}
for some $\widehat{\zeta} \in \widehat{\calH}^{\calP}_M$. \\
\subsubsection{Gluing} 
Suppose we have two manifolds $M_1$ and $M_2$ that share a boundary component $\Sigma$. Then we can glue them together along $\Sigma$ to obtain a new manifold $M = M_1 \cup_{\Sigma} M_2$. The state $\widehat{\psi}_M$ can now be computed from the states $\widehat{\psi}_{M_1}$ and $\widehat{\psi}_{M_2}$ in the following way: Fix polarisations such that $\Sigma \subseteq \de_1M_1$ on $M_1$ and $\Sigma \subseteq \de_2M_2$ on $M_2$. Denote by $\bbA^{\Sigma}$ coordinates on $\Omega^{\bullet}(\Sigma)[1] \subseteq \calB^{\calP}_{\de M_1}$ and by $\bbB^{\Sigma}$ coordinates on $\Omega^{\bullet}(\Sigma)[1] \subseteq \calB^{\calP}_{\de M_2}$. Then we define $\widetilde{\psi}_M$ by 
\[ \widetilde{\psi}_M = \int_{\bbA^{\Sigma},\bbB^{\Sigma}}e^{\frac{i}{\hbar}\int_{\Sigma}\bbB^{\Sigma}\bbA^{\Sigma}}\widehat{\psi}_{M_1}\widehat{\psi}_{M_2}.\]
Again, this integration is defined by a variant of Wick's theorem: The integral of a  term in the product of the states is nonzero if we can contract every $\bbA^{\Sigma}$ with to a $\bbB^{\Sigma}$. In this case, we sum over all possibilities to do so, and every contraction of a $\bbA^{\Sigma}(x)$ with a $\bbB^{\Sigma}(y)$ yields a $\delta^{(2)}_{\de M}(x,y)$. \\ 
One also has to take care of the residual fields: This glued state will usually depend on a non-minimal amount of residual fields, and one can pass to the minimal amount of residual fields by a BV pushforward, yielding the ``correct'' state $\widehat{\psi}_M$. 
\subsubsection{BF-like theories} 
As above, we call ``BF-like'' those theories whose action can be decomposed as $\calS_{BF} + \calS_{\text{int}}$. It is useful to also allow for the free part to consist of several copies of abelian BF theories. One way to do this is to change the space of fields to $\calF_M = (\Omega^{\bullet}(M)\otimes V[1]) \oplus (\Omega^{\bullet}(M)\otimes V^*[1])$ with action 
$$\calS_{M,0} = \int_M \langle \sfB, \dr\sfA \rangle.$$
The discussion above goes through. The only thing that changes in the gauge fixing is that we should replace $\eta$ by $\tilde{\eta}=\eta \otimes \operatorname{id}_V \in \Omega(C^0_2(M))\otimes (V\otimes V^*)$, so that in any basis $\xi_i$ of $V$ with dual basis $\xi^i$ it is given by $$\tilde{\eta}(x_1,x_2) = \eta(x_1,x_2) \delta^i_j \xi_i\otimes\xi^j.$$
\section{Chern-Simons theory as a BF-like theory}
\subsection{Split BV Chern-Simons theory}
Let $\mathfrak{g}$ be a Lie algebra with an ad-invariant pairing $\langle\cdot,\cdot\rangle\colon \mathfrak{g}\times\mathfrak{g}\to\R$, i.e we have for all $x,y,z \in\g$ that $\langle x,[y,z] \rangle =\langle [x,y],z \rangle $. Let $M$ be a 3-manifold, and $\mathsf{C}\in\Omega^{\bullet}(M)\otimes \g[1].$ Then the BV Chern-Simons action is \cite{Cattaneo2014}
$$S[\mathsf{C}]=\int_M \frac{1}{2}\langle\mathsf{C},\dr\mathsf{C}\rangle + \frac{1}{6}\langle \mathsf{C},[\mathsf{C},\mathsf{C}]\rangle,$$  
where for homogeneous elements $A \otimes v, B\otimes w \in \Omega^{\bullet}(M) \otimes \mathfrak{g}$ the bracket  and the pairing are defined by 
$$ [A\otimes v, B \otimes w] = A\wedge B \otimes [v,w] $$ and $$ \langle A \otimes v, B \otimes w \rangle =\langle v, w \rangle A \wedge B $$ respectively. 
Now assume that the Lie Algebra $\g$ admits a splitting $g = V \oplus W$ into maximally isotropic subspaces, i.e. the pairing restricts to 0 on $V$ and $W$ and $\dim V = \dim W = \frac{\dim \mathfrak{g}}{2}$. Then we can identify $W\cong V^*$ via the pairing and decompose $\mathsf{C} = \mathsf{A} + \mathsf{B}$, where $\mathsf{A} \in \Omega^{\bullet}(M) \otimes V[1] $and $\mathsf{B} \in \Omega^{\bullet}(M) \otimes W[1]$. The action decomposes into a ``free'' or ``kinetic'' part 
$$S_{free} = \int_M \frac{1}{2}  \langle\mathsf{C},\dr\mathsf{C}\rangle = \int_M \frac{1}{2}\langle\sfA +\sfB,\dr\sfA + \dr\sfB\rangle = \int_M \frac{1}{2}\langle\sfA,\dr\sfB\rangle +\frac{1}{2}\langle\sfB,\dr\sfA\rangle = \int_M \langle\sfB,\dr\sfA\rangle $$
(where $\langle\sfA,\dr\sfA\rangle=0=\langle\sfB,\dr\sfB\rangle$ by isotropy and we integrate by parts) and an ``interaction'' term 
$$\mathcal{V}\langle\sfA,\sfB\rangle = \frac{1}{6}\langle\sfA+\sfB,[\sfA+\sfB,\sfA+\sfB]\rangle.$$
Hence, the theory is ``BF-like''.
\subsection{Perturbative Expansion}

Let $M$ be a 3-manifold, possibly with boundary. We want to compute the state $\widehat{\psi}_M$. As described above for the BF example, we choose a decomposition of the boundary $\de M= \de_1 M \sqcup \de_2M$ and get a polarisation on the space of boundary fields such that $\mathcal{B}^{\mathcal{P}}_{\de M} = \mathcal{B}_1 \times \mathcal{B}_2 \ni (\bbA,\bbB)$. Decomposing $\sfA = \bbA + \mathsf{a} + \alpha, \sfB = \bbB + \mathsf{b} + \beta$, we can decompose the action as explained in remark \ref{FreeActionSplittingRemark}:
$$ \mathcal{S^{\mathcal{P}}_{\de M}} = \widehat{\mathcal{S}}_{M,0} + \widehat{\mathcal{S}}_{M,\text{pert}} + \mathcal{S}_M^{\text{back}}+\mathcal{S}_M^{\text{source}},$$
where 
\begin{align*}
\widehat{\mathcal{S}}_{M,0} &= \int_M \langle\beta,\dr\alpha\rangle, \\
\widehat{\mathcal{S}}_{M,\text{pert}} &= \int_M \mathcal{V}(\widehat{A},\widehat{B}), \\
 \mathcal{S}_M^{\text{back}}&= - \left(\int_{\de_2M} \langle\bbB,\mathsf{a}\rangle + \int_{\de_1M} \langle\mathsf{b},\bbA\rangle\right),\\ 
 \mathcal{S}_M^{\text{source}} &= - \left(\int_{\de_2M} \langle\bbB,\alpha\rangle + \int_{\de_1M} \langle\beta,\bbA\rangle\right).
\end{align*}
The state is given by 
$$\widehat{\psi}_M = \widehat{\psi}_M(\bbA,\bbB,\mathsf{a},\mathsf{b}) = \int_\mathcal{L} e^{\frac{i}{\hbar}\mathcal{S}^{\mathcal{P}}_M},$$ 
where $\mathcal{L} \ni (\alpha,\beta)$, the gauge-fixing Lagrangian, does not depend (in this case) on the boundary and background fields. By virtue of the above decomposition, we can rewrite this as 
$$\widehat{\psi}_M(\bbA,\bbB,\mathsf{a},\mathsf{b}) = e^{\frac{i}{\hbar}\S^{\text{back}}_M}\int_{\mathcal{L}} e^{\frac{i}{\hbar}\widehat{\S}_{M,0}}e^{\frac{i}{\hbar}\widehat{\S}_{M,\text{pert}}}e^{\frac{i}{\hbar}\S^{\text{source}}_M}.
$$
To do a perturbative (power series) expansion\footnote{Actually, a semiclassical expansion around the classical solution given by the trivial connection},  expand the exponentials 
\begin{align*}
\widehat{\psi}_M(\bbA,\bbB,\mathsf{a},\mathsf{b})  
&= \sum_k\frac{1}{k!}\left(-\frac{i}{\hbar}\right)^k\left(\int_{\de_2M}\langle\bbB,\sfa\rangle +\int_{\de_1M}\langle\sfb\bbA\rangle\right)^k\int_{\mathcal{L}}e^{i\widehat{\S}_{M,0}}\sum_l\frac{1}{l!}\left(\frac{i}{\hbar}\right)^l\left(\int_M\mathcal{V}(\widehat{A},\widehat{B})\right)^l \\
&\sum_m\frac{1}{m!}\left(-\frac{i}{\hbar}\right)^m\left(\int_{\de_2M}\langle\bbB,\alpha\rangle +\int_{\de_1M}\langle\beta,\bbA\rangle\right)^m \\
&= \sum_{k,l,m}\frac{1}{k!l!m!}(-1)^{k+m}\left(\frac{i}{\hbar}\right)^{k+l+m}\left(\int_{\de_2M}\langle\bbB,\sfa\rangle +\int_{\de_1M}\langle\sfb,\bbA\rangle\right)^k \\ 
&\int_{\mathcal{L}}e^{i\widehat{\S}_{M,0}}\left(\int_M\mathcal{V}(\widehat{\sfA},\widehat{\sfB})\right)^l \left(\int_{\de_2M}\langle\bbB,\alpha\rangle +\int_{\de_1M}\langle\beta,\bbA\rangle\right)^m \\
&= \sum_{l,k,m}\frac{1}{k!l!m!}(-1)^{k+m}\left(\frac{i}{\hbar}\right)^{k+l+m}\left(\int_{\de_2M}\langle\bbB,\sfa\rangle +\int_{\de_1M}\langle\sfb,\bbA\rangle\right)^k \\
&\int_{\mathcal{L}}e^{i\widehat{\S}_{M,0}} \left(\int_M\frac{1}{6}\left\langle\widehat{\sfA}+\widehat{\sfB},\left[\widehat{\sfA}+\widehat{\sfB},\widehat{\sfA}+\widehat{\sfB}\right]\right\rangle\right)^l 
\left(\int_{\de_2M}\langle\bbB,\alpha\rangle +\int_{\de_1M}\langle\beta,\bbA\rangle\right)^m .
\end{align*}
Now we choose a basis $\xi_i$ of $V$ and let $\xi^i$ be the corresponding dual basis of $W$. We expand our fields $\sfA = \sfA^i\xi_i, \sfB = \sfB_i\xi^i$ and also their decompositions accordingly, i.e. $\alpha = \alpha^i\xi_i$, and so on. We then get e.g. $\langle\sfB,\dr\sfA\rangle = \sfB_i\dr\sfA^i$. We now want to expand the perturbation term in this basis. For this purpose we make use of the fact that $\langle X, [Y,Z] \rangle = \langle Z, [X,Y] \rangle = \langle Y, [Z,X] \rangle$ for any $X,Y,Z \in \Omega^{\bullet}(M) \otimes \mathfrak{g}[1]$, so we can decompose the interaction term as 
$$ \mathcal{V}(\widehat{\sfA},\widehat{\sfB}) = \frac{1}{6}\langle \widehat{\sfA},[\widehat{\sfA},\widehat{\sfA}]\rangle + \frac{1}{2}\langle \widehat{\sfB},[\widehat{\sfA},\widehat{\sfA}]\rangle + \frac{1}{2}\langle \widehat{\sfA},[\widehat{\sfB},\widehat{\sfB}]\rangle + \frac{1}{6}\langle \widehat{\sfB},[\widehat{\sfB},\widehat{\sfB}]\rangle. $$
Now we make the following simplifying assumption on $\mathfrak{g}$.
\begin{ass} 
The splitting $\mathfrak{g}= V \oplus W$ is actually a splitting into \emph{Lie subalgebras}, i.e. $(\mathfrak{g},V,W)$ is a Manin triple.
\end{ass} 
By isotropy of the subspaces, this implies that the terms $\langle \widehat{\sfA},[\widehat{\sfA},\widehat{\sfA}]\rangle$ and $\langle \widehat{\sfB},[\widehat{\sfB},\widehat{\sfB}]\rangle$ vanish. 
Expanding the perturbation term,  terms of the type $\langle\gamma_1,[\gamma_2,\gamma_3]\rangle$, where $\gamma_i \in \{\sfa,\alpha,\sfb,\beta\}$, each of which can be expanded as $f_{ijk}\gamma^i_1\gamma_2^j\gamma_3^k$. Integration over $\mathcal{L}$ can then be performed using Wick's theorem. Let $\eta$ be an abelian BF propagator on $M$ as discussed above. We exchange integrals over $M,\de_i M $ and $\mathcal{L}$ and get an integrand which is a sum of products of forms $\gamma$. By the Wick theorem, the integral vanishes except for the case where there are precisely as many $\alpha$'s as $\beta$'s, in which case 
$$\int_{\mathcal{L}}e^{i\widehat{\S}_{M,0}}\alpha^{j_1}(x_{1})\cdots\alpha^{j_n}(x_{n})\beta^{k_1}(y_1)\cdots\beta^{k_n}(y_n)=T_M (-i\hbar)^n \sum_{\sigma \in S_n}\delta^{j_{1}k_{\sigma(1)}}\eta(x_1,y_{\sigma(1)})\cdots\delta^{j_{n}k_{\sigma(n)}}
\eta(x_n,y_{\sigma(n)}),$$
where $T_M = \int_{\mathcal{L}}e^{i\widehat{\S}_{M,0}}$. 
\subsection{Feynman graphs and rules}
After integration over $\mathcal{L}$, we can label the terms in the perturbative expansion by graphs as follows. Fix $k,l,m \in \N_0$. We consider graphs $\Gamma$ with three types of vertices: 
\begin{itemize}[-]
\item \emph{Boundary background vertices:} There are $k$ of these distributed on $\de M$. They are labelled by $\bbB\sfa$ if they lie on $\de_2M$ and $\sfb\bbA$ if they lie on $\de_1M$.
\item \emph{Boundary source vertices: } There are $m$ boundary source vertices distributed on $\de M$. They are labelled by $\bbB\alpha$ on $\de_2 M$ and $\bbA\beta$ on $\de_1 M$. Vertices on $\de_2M$ have an arrow tail originating from them, whereas vertices on $\de_1M$ have an arrowhead pointing towards them.
\item \emph{Internal interaction vertices: }There are $l$ internal vertices. They come with three  half-edges which are labelled by $\gamma_i$'s in $\{\sfa,\alpha,\sfb,\beta\}$. These half-edges are either marked as leaves if they are labelled by a background, as an arrow tail if they are labelled by $\alpha$, or an arrowhead if they are labelled by $\beta$ 
\end{itemize} 
If it is possible to connect every arrow tail $\alpha$ to an arrowhead $\beta$ (possibly at the same vertex), then the graph resulting from this procedure is called an  \emph{admissible graph}. To such a graph we can associate a functional on the space of boundary fields as follows: 
\begin{itemize}[-] 
\item For every background boundary vertex, multiply by  $(-i/\hbar)$ times the label and integrate over the corresponding boundary point.
\item For every internal vertex  multiply by $(-i/\hbar)$ times the correct structure constants (specified by the half-edge labels) and  integrate over $M$.
\item For every leaf, multiply by the corresponding background field evaluated at the point. 
\item For every arrow between vertices in different positions $i \neq j$, with tail labelled by $\alpha^k$ and head $\beta_l$, multiply by a propagator $(-i\hbar)\delta^k_l\eta(x_i,x_j)$. 
\item For every short loop (also called tadpole), i.e. an arrow issueing and ending at the same vertex $i$ ,with tail labelled by $\alpha^k$ and head $\beta_l$, multiply by  $(-i\hbar)\delta^k_l\alpha(x_i)$, where $\alpha \in \Omega^2(M)$ is a so-called ``tadpole form''.\footnote{These contributions can be ignored if the Lie algebra is unimodular (i.e. the structure constants satisfy $f^i_{ik} = 0$) or the Euler characteristic of $M$ is 0. We will restrict ourselves to these cases.}
\item For every source boundary vertex, we multiply by $(-i/\hbar)$ times the corresponding boundary field and integrate over the corresponding boundary point. 
\end{itemize}
We denote the result by $\widehat{\psi}_{\Gamma}$.
Denoting the set of all admissible graphs for $k,l,m$ by $\Lambda_{k,l,m}$, we get
$$ \widehat{\psi}_M(\bbA,\bbB,\mathsf{a},\mathsf{b}) = T_M\sum_{k,l,m}\sum_{\Gamma\in\Lambda_{k,l,m}}\widehat{\psi}_{\Gamma}.$$
\begin{rem}
We can factor out the non-interacting diagram parts (background boundary vertices and source boundary vertices connecting to other source boundary vertices). This will yield a prefactor of $e^{\frac{i}{\hbar}\calSeff_0}$ where $\calSeff_{0}$ is the free effective action 
\begin{equation}
\calSeff_0 = - \left(\int_{\de_2M} \langle\bbB,\mathsf{a}\rangle + \int_{\de_1M} \langle\mathsf{b},\bbA\rangle\right) - \int_{\de_2M \times \de_1M} \pi_1^*\bbB_i\eta\bbA^i \label{SeffFree}\end{equation} 
i.e. the effective action of the unperturbed theory. 
\end{rem}
The remaining interaction diagrams have $l \geq 1$ internal vertices and $m \leq 3l$ boundary vertices. Denoting the set of admissible interaction diagrams by $\Lambda^{int}_{l,m}$, the above expression becomes 
$$ \widehat{\psi}_M(\bbA,\bbB,\mathsf{a},\mathsf{b}) = T_Me^{\frac{i}{\hbar}\calSeff_{0}} \left(1 + \sum_{l=1}^{\infty}\sum_{m=0}^{3l}\sum_{\Gamma\in\Lambda^{int}_{l,m}}
\widehat{\psi}_{\Gamma}\right).$$
Our goal is now to give an asymptotic expansion of the state of the form 
$$\widehat{\psi}_M(\bbA,\bbB,\mathsf{a},\mathsf{b}) =  T_Me^{\frac{i}{\hbar}\calSeff_M}\sum_{j\geq 1}\hbar^jR_j, $$
where $\calSeff_M$ is the so-called \emph{tree effective action}, i.e the sum of all diagrams whose underlying graphs are trees, and $R_j$ denotes the sum of all diagrams that contain at least one loop. 
\section{Split Chern-Simons theory on the solid torus}
In this section we compute a first approximation for the state on the solid torus $K := D \times S^1$ with boundary $\de M = S^1 \times S^1 =: \T^2$. Here we think of $D = \{ z \in \C, |z| \leq 1 \}$ as the closed unit disk in the complex plane. This is not just a simple exercise: Note that since the quantum BV-BFV formalism allows also for the gluing of states, given a state on the solid torus one can compute it also for any manifold that can be glued together from tori (namely, these include all lens spaces). \\
Since the boundary $\T^2$ is connected, there are only two possible choices for $\de_1M$ and $\de_2M$, we choose $\de_1M := \de M$ and $ \de_2M := \emptyset $.
This leads to the following space of backgrounds:
\begin{align*}
\mathcal{V}_M &= \HB_{D1}(M)[1]\otimes V \oplus\HB_{D2}(M)[1] \otimes W = \HB(M,\de M)[1] \otimes V \oplus \HB(M) \otimes W \\ &\cong (\HB(D,\de D) \otimes \HB(S^1)) \otimes V \oplus \HB(S^1))[1] \otimes W.\\
\end{align*}
Let $\mu$ be a normalised generator of $\HB(D,\de D)$, i.e $\int_D\mu = 1$. Denoting $t$ the coordinate on $S^1$, we get that $\chi_1 = \mu dt, \chi_2 = \mu  $ is a basis of $\HB_{D1}(M)[1]$, with dual basis $\chi^1 = 1, \chi^2 = dt$ of $\HB_{D2}(M)[1]$. We can then expand
\begin{align*}
\mathsf{a}^i &= z^{1i}\mu dt + z^{2i}\mu,\\ \mathsf{b}_i&= z_{1i}^+1 + z_{2i}^+dt.
\end{align*}
The canonical BV Laplacian on $\mathcal{V}_M$ is then given by $$\Delta_{\mathcal{V}_M} = -\left(\frac{\de}{\de z^1}\frac{\de}{\de z_1^+} + \frac{\de}{\de z^2}\frac{\de}{\de z_2^+}\right).$$

\subsection{Effective Action on the solid torus} 
Assume as above that $\mathfrak{g} = V \oplus W$ is a Manin triple, i.e
\begin{itemize}
\item $V \cong W^*$ as vector spaces 
\item $V,W$ Lie algebras.
\end{itemize}
Introduce bases $\xi_1,\ldots,\xi_n$ of $V$, $\xi^1,\ldots,\xi^n$ of $W$ such that $\langle \xi_i, \xi^j \rangle = \delta_i^j$. We introduce structure constants in these bases: $[\xi_i,\xi_j]_V = f^k_{ij}\xi_k, [\xi^i,\xi^j]=g_k^{ij}$. 
We can then also decompose the fields 
\begin{align*}
\sfB &= \sfB_i\xi^i = \sfb_i\xi^i + \beta_i\xi^i + \bbB_i\xi^i, \\
\sfA &= \sfA^i\xi_i = \sfa^i\xi_i + \alpha^i\xi_i + \bbA^i\xi_i.
\end{align*}
The fact we have a Manin triple means that in terms of the structure constants we have
\begin{equation} f^k_{ij}g_k^{lm} = f^l_{ik}g^{km}_j - f^l_{jk}g^{km}_i + f^m_{ik}g^{lk}_j - f^m_{jk}g^{lk}_i. \label{bialgebra}\end{equation}
We now want to compute an approximation to the tree effective action by considering tree diagrams that have at most two interaction vertices and at most two boundary vertices. \\ 
We will proceed by the number of interaction vertices. There is only a single connected diagram with no interaction vertices, consisting of a single point on the boundary. It yields the free effective action \eqref{SeffFree} for $\de_2M = \emptyset$, namely 
$$\Seff_0 =  - \int_{\de_1M} \sfb_k\bbA^k.$$
\subsubsection{1-point contribution}
Let us continue with diagrams containing a single interaction vertex. It is now important that the solid torus has zero Euler characteristic, so we do not need to consider tadpoles. Since there can be no arrows issuing from $\de_1M$, diagrams with a half-edge labelled by $\beta$ at the interaction point are not admissible. Also notice that $\sfa \wedge \sfa = 0$ (it is a 4-form on a 3-manifold). In the end, there are only three contributing diagrams:
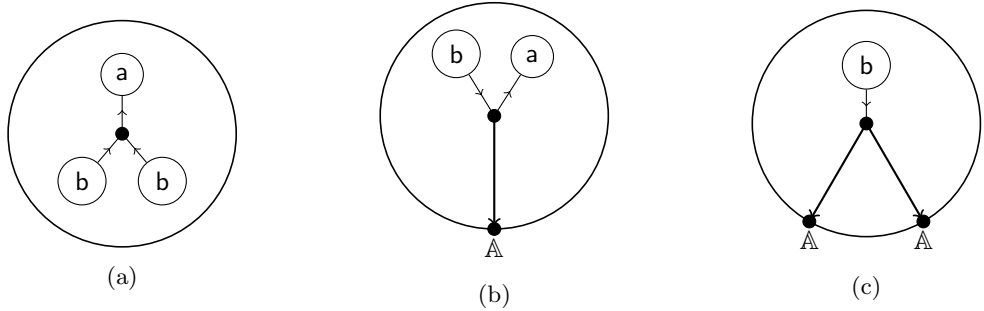
\begin{figure}[h]
\centering
\begin{subfigure}{0.3\textwidth}
\centering
\begin{tikzpicture}[scale=1]

  \node[shape=coordinate] (O) at (0,0) {};

 \tkzDefPoint(1.5,0){A}
  \node (resb) at (0,0.5) [shape=circle,above,draw,minimum size=1pt] {$\sfa$}
edge[middlearrow={<}] (O);
\node (resa) at (-0.3,-0.4) [shape=circle,below left,draw,minimum size=1pt] {$\sfb$}
edge[middlearrow={>}] (O);
\node (resa) at (0.3,-0.4) [shape=circle,below right,draw,minimum size=1pt] {$\sfb$}
edge[middlearrow={>}] (O);
  \tkzDrawCircle(O,A)
  \tkzDrawPoint[color=black,fill=black,size=12](O)
\end{tikzpicture}
\caption{}
\end{subfigure}
\begin{subfigure}{0.3\textwidth}
\centering
\begin{tikzpicture}[scale=1]

  \node[shape=coordinate] (O) at (0,0) {};

  \node[shape=coordinate,label=below:{$\bbA$}] (bdry1) at (0,-1.5) {$\bbA$}
  edge[<-,shorten <= 1.5pt,thick] (O);
\node (resb) at (-0.5,0.5) [shape=circle,above,draw,minimum size=1pt] {$\sfb$}
edge[middlearrow={>}] (O);
\node (resa) at (0.5,0.5) [shape=circle,above,draw,minimum size=1pt] {$\sfa$}
edge[middlearrow={<}] (O);
  \tkzDrawCircle(O,A)
  \tkzDrawPoints[color=black,fill=black,size=12](O,bdry1)
\end{tikzpicture}
\caption{}
\end{subfigure}
\begin{subfigure}{0.3\textwidth}
\centering
\begin{tikzpicture}[scale=1.5]
\node[shape=coordinate] (O) at (0,0) {};
\node (resb) at (0,0.3) [shape=circle,above,draw,minimum size=1pt] {$\sfb$}
edge[middlearrow={>}] (O);
\node[shape=coordinate,label=below:{$\bbA$}] (bdry1) at (canvas polar cs: angle=-120, radius=1cm) {$\bbA$}
edge[<-,shorten <= 1.5pt,thick] (O);
\node[shape=coordinate,label=below:{$\bbA$}] (bdry2) at (canvas polar cs: angle=-60, radius=1cm) {$\bbA$}
edge[<-,shorten <= 1.5pt,thick] (O);

\tkzDrawCircle(O,bdry1)
  \tkzDrawPoints[color=black,fill=black,size=12](O,bdry1,bdry2)

\end{tikzpicture}
\caption{}
\end{subfigure}
\caption{Graphs in the solid torus (depicted in a cross-section) with 1 interaction vertex. A bullet denotes a point we integrate over, a long arrow denotes a propagator.}
\end{figure}

\begin{enumerate}[a)]
\item The single interaction vertex with three leaves labelled by $\sfa, \sfb$ and $\sfb$, corresponding to 
$$  \Seff_1 :=\frac{1}{2}\int_M\langle \sfa, [\sfb,\sfb]\rangle. $$
We should explain some notation. Wew denote by $C_{m,n}(M,\de M)$  (a suitable compactification of) the configuration space of $m$ points in the bulk and $n$ in the boundary. It comes with natural projections $$\pi_i\colon C_{m,n}(M,\de M) \to \begin{cases} &M \quad i\leq m  \\ &\de M \quad i\geq m\end{cases} $$ and $$\pi_{ij}\colon C_{i,j}(M,\de M) \to \begin{cases} &C_2(M) \quad i,j\leq m \\ &C_{1,1}(M,\de M)  \quad i\leq m, j \leq n \\ &C_2(\de M) \quad i,j\geq m\end{cases}. $$ 
By writing $\gamma_i$ resp. $\gamma_{ij}$ we mean the pullback of $\gamma$ under the corresponding projection.
\item The single interaction vertex with two leaves labelled $\sfb$ and $\sfa$ and an arrow connecting to a boundary source vertex $\beta\bbA$. It evaluates to 
$$ \Seff_2 := - \int_{C_{1,1}(M,\de M)}  f^i_{jk} \sfb_{1,i}\sfa_1^j\eta_{12}\bbA_2^k. $$ 
\item The single interaction vertex with a leaf labelled by $\sfb$ and two arrows connecting to two different boundary source vertices. This evaluates to $$S_{\text{eff},3}:= \frac{1}{2}\int_{C_{1,2}(M,\de_1M)}f^i_{jk}b_{1,i}\eta_{12}\eta_{13}
\bbA_2^j\bbA_3^k.$$
\end{enumerate}
\subsubsection{2-point contribution}
Now we consider tree diagrams with two interaction vertices. Since the diagrams have to be connected, there has to be at least one arrow between the vertices. Since we are only considering trees, there is exactly one arrow between them. Also, we are considering only diagrams that have at most two boundary vertices. The diagrams in figure \ref{fig:Diagrams2} below show the admissible graphs in the relevant degrees (admissible graphs with no boundary vertices all evaluate to 0 because of property \ref{Koiequalzero}) \begin{figure}[h]
\centering
\begin{subfigure}{0.3\textwidth}
\centering
\begin{tikzpicture}[scale=1]
\coordinate (O) at (0,0);
 \coordinate (bulk1) at (-0.5,0) {};
 \coordinate (bulk2) at (0.5,0.0) {}
 edge[<-, shorten <= 1.5pt,thick] (bulk1);
\node[shape=coordinate,label=below:{$\bbA$}] (bdry1) at (canvas polar cs: angle=-60, radius=1.5cm) {$\bbA$}
edge[<-,shorten <= 1.5pt,thick] (bulk2);
\node (resb) at (-0.5,0.4) [shape=circle,above left,draw] {$\sfb$}
edge[middlearrow={>}] (bulk1);
\node (resa2) at (-0.5,-0.4) [shape=circle,below left,draw] {$\sfa$}
edge[middlearrow={<}] (bulk1);
\node (resa) at (0.5,0.4) [shape=circle,above right,draw] {$\sfb$}
edge[middlearrow={>}] (bulk2);

\tkzDrawCircle(O,bdry1)
  \tkzDrawPoints[color=black,fill=black,size=12](bulk1,bulk2,bdry1,bdry2)
\end{tikzpicture}
\caption{}
\end{subfigure}
\begin{subfigure}{0.3\textwidth}
\centering
\begin{tikzpicture}[scale=1]
\coordinate (O) at (0,0);
 \coordinate (bulk1) at (-0.5,0) {};
 \coordinate (bulk2) at (0.5,0.0) {}
 edge[<-, shorten <= 1.5pt,thick] (bulk1);
\node[shape=coordinate,label=below:{$\bbA$}] (bdry1) at (canvas polar cs: angle=-60, radius=1.5cm) {$\bbA$}
edge[<-,shorten <= 1.5pt,thick] (bulk2);
\node (resb) at (-0.5,0.4) [shape=circle,above left,draw] {$\sfb$}
edge[middlearrow={>}] (bulk1);
\node (resa2) at (-0.5,-0.4) [shape=circle,below left,draw] {$\sfa$}
edge[middlearrow={<}] (bulk1);
\node (resa) at (0.5,0.4) [shape=circle,above right,draw] {$\sfa$}
edge[middlearrow={<}] (bulk2);

\tkzDrawCircle(O,bdry1)
  \tkzDrawPoints[color=black,fill=black,size=12](bulk1,bulk2,bdry1,bdry2)
\end{tikzpicture}
\caption{}
\end{subfigure}
\begin{subfigure}{0.3\textwidth}
\centering
\begin{tikzpicture}[scale=1]
\coordinate (O) at (0,0);
 \coordinate (bulk1) at (-0.5,0) {};
 \coordinate (bulk2) at (0.5,0.0) {}
 edge[<-, shorten <= 1.5pt,thick] (bulk1);
 \node[shape=coordinate,label=below:{$\bbA$}] (bdry1) at (canvas polar cs: angle=-120, radius=1.5cm) {$\bbA$}
edge[<-,shorten <= 1.5pt,thick] (bulk1);
\node[shape=coordinate,label=below:{$\bbA$}] (bdry2) at (canvas polar cs: angle=-60, radius=1.5cm) {$\bbA$}
edge[<-,shorten <= 1.5pt,thick] (bulk2);
\node (resb) at (-0.5,0.4) [shape=circle,above left,draw] {$\sfb$}
edge[middlearrow={>}] (bulk1);
\node (resa) at (0.5,0.4) [shape=circle,above right,draw] {$\sfa$}
edge[middlearrow={<}] (bulk2);
\tkzDrawCircle(O,bdry1)
  \tkzDrawPoints[color=black,fill=black,size=12](bulk1,bulk2,bdry1,bdry2)
\end{tikzpicture}
\caption{}
\label{fig:diagram1}
\end{subfigure}

\begin{subfigure}{0.3\textwidth}
\centering
\begin{tikzpicture}[scale=1]
\coordinate (O) at (0,0);
 \coordinate (bulk1) at (-0.5,0) {};
 \coordinate (bulk2) at (0.5,0.0) {}
 edge[<-, shorten <= 1.5pt,thick] (bulk1);
 \node[shape=coordinate,label=below:{$\bbA$}] (bdry1) at (canvas polar cs: angle=-120, radius=1.5cm) {$\bbA$}
edge[<-,shorten <= 1.5pt,thick] (bulk1);
\node[shape=coordinate,label=below:{$\bbA$}] (bdry2) at (canvas polar cs: angle=-60, radius=1.5cm) {$\bbA$}
edge[<-,shorten <= 1.5pt,thick] (bulk2);
\node (resb) at (-0.5,0.4) [shape=circle,above left,draw] {$\sfb$}
edge[middlearrow={<}] (bulk1);
\node (resa) at (0.5,0.4) [shape=circle,above right,draw] {$\sfb$}
edge[middlearrow={>}] (bulk2);
\tkzDrawCircle(O,bdry1)
  \tkzDrawPoints[color=black,fill=black,size=12](bulk1,bulk2,bdry1,bdry2)
\end{tikzpicture}
\caption{}
\end{subfigure}
\begin{subfigure}{0.3\textwidth}
\centering
\begin{tikzpicture}[scale=1]
\node[shape=coordinate] (O) at (0,0) {};
\coordinate (bulk1) at (0,0.2) {};
\coordinate (bulk2) at (0,-0.6) {}
 edge[<-, shorten <= 1.5pt,thick] (bulk1);
\node (resb) at (-0.3,0.4) [shape=circle,above left,draw,minimum size=1pt] {$\sfb$}
edge[middlearrow={>}] (bulk1);
\node (resa) at (0.3,0.4) [shape=circle,above right,draw,minimum size=1pt] {$\sfa$}
edge[middlearrow={<}] (bulk1);
\node[shape=coordinate,label=below:{$\bbA$}] (bdry1) at (canvas polar cs: angle=-120, radius=1.5cm) {$\bbA$}
edge[<-,shorten <= 1.5pt,thick] (bulk2);
\node[shape=coordinate,label=below:{$\bbA$}] (bdry2) at (canvas polar cs: angle=-60, radius=1.5cm) {$\bbA$}
edge[<-,shorten <= 1.5pt,thick] (bulk2);

\tkzDrawCircle(O,bdry1)
  \tkzDrawPoints[color=black,fill=black,size=12](bulk1,bulk2,bdry1,bdry2)

\end{tikzpicture}
\caption{}
\label{fig:diagram2}
\end{subfigure}
\caption{Graphs with 2 interaction vertices. A bullet denotes a point we integrate over, long arrow denotes a propagator.}
\label{fig:Diagrams2}
\end{figure}
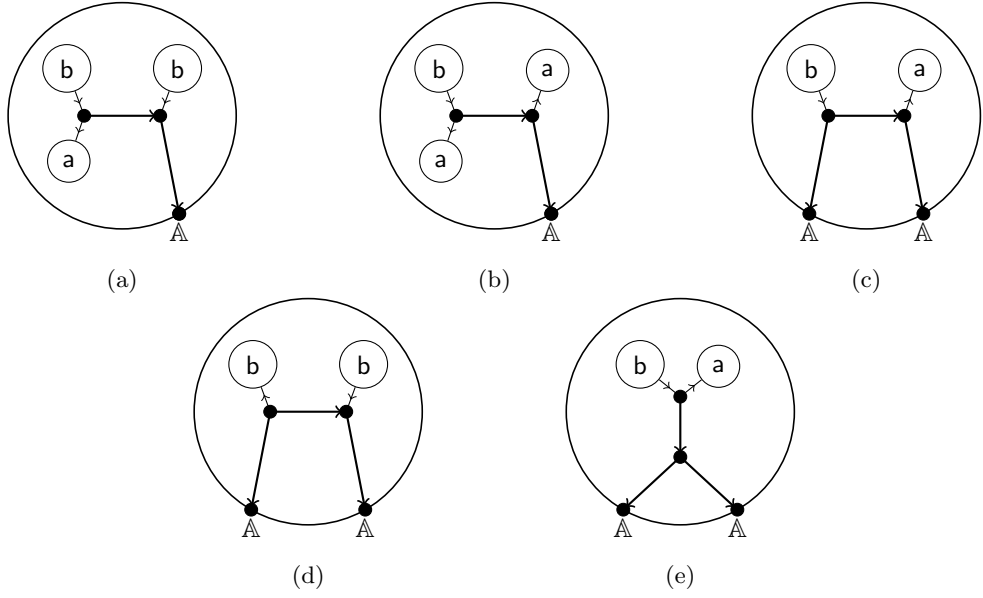
We will discuss the results below.

\subsubsection{Performing integration over M}\label{Integration}
We now want to perform the integration over the bulk points. There are two possibilities to proceed: 
\begin{enumerate}
\item One constructs an explicit propagator on $M$ and computes the integrals analytically.
\item One analyses how the resulting form on the boundary behaves under de Rham differential and integration of points, and picks a form which is a product of propagators and representatives of cohomology on the boundary that has the same properties. Since only these properties enter into the proof of the mQME, this produces a valid state. We will discuss this procedure and the question of uniqueness in more depth in a future paper. 
\end{enumerate}

With the second approach, choosing a propagator satisfying also \ref{Koiequalzero} and \ref{KoKequalzero}, one can see that the only contributions from two-point diagrams \ref{fig:diagram1} and \ref{fig:diagram2}. Denoting the results by $\Seff_4$ and $\Seff_5$ respectively, we obtain
\begin{align*}
 \Seff_0 &= -z_{1,k}^+\int_{\de_1M}\bbA^k -z_{2,k}^+\int_{\de_1M}dt\bbA^k,\\
\Seff_1 &= \frac{1}{2}g_i^{jk}(z^{1i}z^+_{1j}z^+_{1k}+2z^{2i}z^+_{1j}z^+_{2k}),\\
\Seff_2  &= f^i_{jk}z_{1i}^+z^{2j}\int_{\de_1M}d\theta\bbA^k+ f^i_{jk}(z^+_{1i}z^{1j}-z^+_{2i}z^{2j})\int_{\de_1M} dtd\theta \bbA^k,\\
\Seff_3 &= \frac{1}{2}f^i_{jk}z_{1i}^+\int_{C_2(\de_1M)}\eta^T_{12}\bbA^j_1\bbA^k_2, \\ &+\frac{1}{2}f^i_{jk}z_{2i}^+ \int_{C_2(\de_1M)}\eta^T_{12}\frac{dt_1+dt_2}{2}\bbA^j_1\bbA^k_2, \\
\Seff_4 &= f^i_{jk}f^j_{lm}z_{1i}^+z^{2l}\int_{C_2(\de_1M)}d\theta_1
\eta^T_{12}\bbA^k_1\bbA^m_2 +f^i_{jk}f^j_{lm}(z^+_{1i}z^{1l}-z^+_{2i}z^{2l}) \int_{C_2(\de_1M)}dt_1d\theta_1\eta^T_{12}\bbA^k_1\bbA^k_m, \\
\Seff_5 &= f^i_{jk}f^j_{lm}z_{1i}^+z^{2j}\int_{C_2(\de_1M)}d\theta_1
\eta^T_{12}\bbA^l_1\bbA^m_2 +f^i_{jk}f^j_{lm}(z^+_{1i}z^{1j}-z^+_{2i}z^{2j}) \int_{C_2(\de_1M)}dt_1d\theta_1\eta^T_{12}\bbA^l_1\bbA^m_2,
\end{align*}

where $t$ denotes the parallel (longitudinal) coordinate on the torus, $\theta$ the meridian (i.e. in the solid torus $[\dr\theta] = 0$) and $\eta^T$ is a propagator for abelian BF theory on the torus. 
\subsection{mQME}
Our goal in this Section is to prove  the modified Quantum Master Equation $$(\hbar^2\Delta + \Omega)e^{\frac{i}{\hbar}S_{\text{eff}}} = 0,$$
ignoring terms of order $\hbar$, more than two boundary vertices or more than second power in the interaction. Here $\Omega$ is given by the standard quantisation of $$S^{\de} = \int_{\de M} \langle \sfB, \dr\sfA \rangle + \frac{1}{2}\langle \sfB,[\sfA,\sfA]\rangle +\frac{1}{2}\langle \sfA,[\sfB,\sfB]\rangle, $$ 
which (on the solid torus) is 
\begin{align*}
\Omega_{\text{st}} &= -\ii\hbar\int_{\de_1M} \dr \bbA^k  \frac{\delta}{\delta \bbA^k} +\frac{1}{2}g_a^{bc}\int_{\de_1M}-\hbar^2 \bbA^a \frac{\delta}{\delta \bbA^b}\frac{\delta}{\delta \bbA^c}  
+\frac{1}{2}f^a_{bc}\int_{\de_1M}-\ii\hbar\bbA^b\bbA^c\frac{\delta}{\delta\bbA^a}.
\end{align*}
\begin{rem} The second term containing two derivatives yields possibly singular results when applied to a single term in the effective action. Therefore the two derivatives are allowed to act only on different terms in a product of terms of the effective action. With this regularisation one can also check that $\Omega_{st}^2 = 0$.
\end{rem}
One can check that $\Delta S_{\text{eff}} = 0$ and therefore $(\hbar^2\Delta + \Omega)e^{\frac{i}{\hbar}\Seff}  = -\frac{1}{2}(\Seff ,\Seff)e^{\frac{\ii}{\hbar}\Seff} + \Omega e^{\frac{\ii}{\hbar}\Seff}. $ 
So we should check that $\frac{1}{2}(\Seff ,\Seff)e^{\frac{\ii}{\hbar}\Seff} = \Omega e^{\frac{\ii}{\hbar}\Seff}$ up to higher order corrections. \\
\subsubsection{BV Bracket}
Let us compute first $(\Seff ,\Seff)$. Abbreviating $\Seff_i=:S_i$, we get that 
$(\Seff ,\Seff) = \sum_i (S_i,S_i) + 2\sum_{i<j} (S_i,S_j)$. \\
We have that $(z^+_{1i},z^{1j}) = \delta_{ij} = -(z^+_{2i},z^{2j})$, and all other brackets vanish. \\
Since $S_0$ and $S_3$ only contain $z^+$ variables, we get that $(S_0,S_0) = (S_3,S_3) = (S_0,S_3) = 0$. Also, $(S_2,S_3)$ contains three boundary fields, so we neglect it. The same is true for any bracket of $S_4$ with the rest, except $(S_1,S_4)$, which is third power in the structure constants. So the only contributing brackets are $(S_0,S_1),(S_0,S_2)$, 
 $(S_1,S_1),(S_1,S_2),(S_1,S_3)$ and $(S_2,S_2)$. 
 \subsubsection{$\Omega$ part}
Now let us compute $\Omega_{\text{st}}e^{\frac{\ii}{\hbar}\Seff}.$ At first, we will consider only contributions of order 0 in $\hbar$ and less than two $\bbA's$. Let us split $\Omega$ into the following 3 terms: 
\begin{align*}
\Omega_{0} &:= -\ii\hbar\int_{\de_1M} \dr \bbA^k  \frac{\delta}{\delta \bbA^k}, \\ 
\Omega_1 &:= -\frac{\ii\hbar}{2}f^a_{bc}\int_{\de_1M}\bbA^b\bbA^c\frac{\delta}{\delta\bbA^a}, \\
\Omega_2 &:=-\frac{\hbar^2}{2}g_a^{bc}\int_{\de_1M} \bbA^a \frac{\delta}{\delta \bbA^b}\frac{\delta}{\delta \bbA^c}. \\
\end{align*}
By the usual rules of derivatives we will have 
$$\Omega_{st}e^{\frac{\ii}{\hbar}\Seff}= \left(\left(\Omega_0+\Omega_1\right)\frac{i}{\hbar}\Seff+\Omega_2\left(\frac{i}{\hbar}\right)^2\frac{1}{2}(\Seff)^2\right)e^{\frac{\ii}{\hbar}\Seff}.$$
Let us look at the linear term first. Notice that $\Omega_0(S_0) = \Omega_0(S_1) = \Omega_0(S_2) = 0$, since we can integrate by parts, and the forms appearing in these integrals are closed. Also, since we are ignoring terms with more than two boundary fields, and $\Omega_1(S_1)=0$, we only need to consider $\Omega_1(S_0)$ and $\Omega_1(S_2)$. Now we need to consider $\Omega_2 \left(\frac{\ii}{\hbar}\right)^2\frac{1}{2!}(\Seff)^2$. Since $\Omega_2$ removes one $\bbA$, but adds one power in the interaction, we have to consider terms in $(\Seff)^2$ with two or three $\bbA$'s and at most first power in the the interaction. One can easily check that the only products to consider are $S_0^2,S_0S_2$ and $S_0S_3$. We therefore arrive at the following 
\begin{prop}  
To prove the mQME in the chosen degrees one can equivalently prove that 
\begin{align}(S_0,S_1)&+(S_0,S_2)+\frac{1}{2}(S_1,S_1)+(S_1,S_2)+(S_1,S_3)+\frac{1}{2}(S_2,S_2) \\
&= \frac{i}{\hbar}\left(\Omega_0(S_3) + \Omega_0(S_4) + \Omega_0(S_5) + \Omega_1(S_0)+\Omega_1(S_2)\right)+\frac{1}{2}\left(\frac{i}{\hbar}\right)^2\Omega_2(S_0^2 + 2S_0S_2 + 2S_0S_3).\end{align} \label{redmQME}
\end{prop}
This can be shown in a direct computation, which we as follows.
\begin{lem}\label{lemma1}
The following identities hold: 
\begin{enumerate}[i)]
\item $(S_0,S_1) = \frac{1}{2}\left(\frac{i}{\hbar}\right)^2\Omega_2(S_0^2),$
\item $(S_1,S_1) = 0,$
\item $(S_0,S_2) = \frac{i}{\hbar}\left(\Omega_0(S_3) + \Omega_1(S_0)\right),$
\item $(S_1,S_2) = \left(\frac{i}{\hbar}\right)^2\Omega_2(S_0S_2),$
\item $(S_1,S_3) = \left(\frac{i}{\hbar}\right)^2\Omega_2(S_0S_3),$
\item $(S_2,S_2) = \frac{i}{\hbar}\left(\Omega_0(S_4) + \Omega_0(S_5) + \Omega_1(S_2)\right).$
\end{enumerate}
\end{lem}
\begin{cor}
The state defined by $\widehat{\psi} = e^{\frac{i}{\hbar}\Seff}$ satisfies the mQME on the solid torus at zeroth order in $\hbar$, considering terms with at most two boundary fields and at most second order in the interaction. 
\end{cor}
\subsection{Change of data}
Now we will analyse how the state behaves under an infinitesimal change of gauge-fixing, i.e. the representatives of cohomology and the propagator. Such a change can be described by the action of a vector field $X$ on $M$ on these forms by the Lie derivative 
$$\dot{\chi}_i = L_X\chi_i, \dot{\chi}^i = L_X\chi^i, \dot{\eta} = L_X\eta$$ 
(we will always write $X$ to mean the vector field $(X,\ldots,X) \in TM\oplus \cdots \oplus TM \cong T(M\times\cdots\times M)$). Clearly we have 
$$ \frac{d}{dt}\widehat{\psi} = \frac{i}{\hbar}\frac{d}{dt}(\Seff)e^{\frac{i}{\hbar}\Seff}.$$ 
\begin{prop}\label{changeofdata}
If we expand $\Seff$ as a sum of terms of the form 
$$\Seff=\sum\int_{C_n(\de_1M)}\gamma\pi^*_1\bbA\cdots\pi^*_n\bbA,$$ then its time derivative is given by 
$$\frac{d}{dt}(\Seff)=\sum\int_{C_n(\de_1M)}(L_{X^{\de}}\gamma)\pi^*_1\bbA\cdots\pi^*_n\bbA,$$
where $X^{\de}$ denotes restriction of $X$ to the boundary. 
\end{prop}
\begin{proof}
$\Seff$ is a sum of terms of the form 
$$\int_{C_{m,n}(M,\de_1M)}\widehat{\gamma}\pi^*_1\bbA\cdots\pi^*_n\bbA,$$ where $\widehat{\gamma}$ is a product of background fields and propagators on $M$. Since $L_X$ is a derivation, we have $\frac{d}{dt}\widehat{\gamma} = L_X\widehat{\gamma}$. But the Lie derivative commutes with the integration over the bulk vertices, so we have proved the statement.  
\end{proof}
We are now going to define a state $\zeta$ such that $$(\hbar^2\Delta + \Omega)(\widehat{\psi}\zeta) = \frac{d}{dt}\widehat{\psi}$$
for our example on the torus. Namely, we define $\gamma_i \in \Omega^{k_i}(C_{n_i}(\de_1M))$ by 
$$ \Seff= \sum_i F_i(f,g,z,z^+)_{j_1\cdots\j_{n_i}}\int_{C_{n_i}(\de_1M)}\gamma_i\pi^*_1\bbA^{j_1}\cdots\pi^*_{n_i}\bbA^{j_{n_i}}.$$
Then $\zeta$ is defined by 
$$ \zeta = \sum_iF_i(f,g,z,z^+)_{j_1\cdots\j_{n_i}} \int_{C_{n_i}(\de_1M)}(\iota_{X^{\de}}\gamma_i)\pi^*_1\bbA^{j_1}\cdots\pi^*_{n_i}\bbA^{j_{n_i}} ,$$
i.e. we replace every differential form $\gamma_i$ by its contraction with $X$. 
\begin{prop}
For the change of data described above and the effective action described in the last paragraph, we have that 
$$(\hbar^2\Delta + \Omega)(\widehat{\psi}\zeta) = \frac{d}{dt}\widehat{\psi}$$
at zeroth order in $\hbar$, considering only terms of at most two boundary fields and at most second power in the interaction. 
\end{prop} 
\begin{proof}[Sketch of the proof]
We have that $$\Delta((\widehat{\psi}\zeta)) = \Delta(\widehat{\psi})\zeta \pm \widehat{\psi}\Delta(\zeta) \pm (\psi,\zeta) = \Delta(\widehat{\psi})\zeta  \pm (\psi,\zeta),$$ 
since $\Delta(\zeta) = 0$. On the other hand, using that $\Omega_0$ and $\Omega_1$ are first-order differential operators and $\Omega_2$ is a second-order differential operator,
\begin{align*}
\Omega(\widehat{\psi}\zeta) &= \Omega_0(\widehat{\psi}\zeta) + \Omega_1(\widehat{\psi}\zeta) + \Omega_2(\widehat{\psi}\zeta) \\
&= \Omega_0(\widehat{\psi})\zeta + \widehat{\psi}\Omega_0(\zeta) + \Omega_1(\widehat{\psi})\zeta + \widehat{\psi}\Omega_1(\zeta) + \Omega_2(\widehat{\psi})\zeta + \widehat{\psi}\Omega_2(\zeta) + (\widehat{\psi}\zeta)' \\
&= \Omega(\widehat{\psi})\zeta + \widehat{\psi}\Omega(\zeta) + (\widehat{\psi}\zeta)',
\end{align*}
where $(\widehat{\psi}\zeta)'$ denotes the term where one derivative in $\Omega_2$ acts on $\widehat{\psi}$ and the other acts on $\zeta$. 
By the mQME, terms where $\Delta$ and $\Omega$ act on $\psi$ only cancel. Let us look first at the term where $\Omega$ acts on $\zeta$ only. After integrating by parts, $\Omega_0(\zeta)$ replaces $\iota_{X^{\de}}\gamma_i$ by $\dr\iota_{X^{\de}}\gamma_i$, plus contributions from the boundary of the configuration space. As in the proof of the mQME, those are cancelled by $\Omega_1(\zeta)$. Since $\Omega_2$ can only act on products of terms, $\Omega_2(\zeta) = 0$. Next, notice that by properties of BV brackets and derivatives we have $$(\psi,\zeta)= (\Seff,\zeta)\psi \qquad \text{ and } \qquad (\psi\zeta)' = (\Seff\zeta)'\psi. $$
We are left to prove that $(\Seff,\zeta) + (\Seff\zeta)'$ produces all the terms of the form $\iota_{X^{\de}}\dr\gamma$, then the result follows from 
Proposition \ref{changeofdata} and Cartan's magic formula. We summarise this as follows.
\begin{lem}\label{lemma2}
Let $S_i$ be the parts of the effective action as above. Denote by $\iota_{X^{\de}}S_i, \dr S_i$ the operation of replacing all differential forms $\gamma$ appearing in $S_i$ by $\iota_{X^{\de}}\gamma$ or $\dr\gamma$ respectively. Then the following identities hold: 
\begin{align}\Omega_2(S_0\iota_{X^{\de}}S_0) &= (S_1,\iota_{X^{\de}}S_0), \\
\Omega_2(S_0\iota_{X^{\de}}S_2)+\Omega_2(S_2\iota_{X^{\de}}S_0) &= (S_1,\iota_{X^{\de}}S_2),\\
\Omega_2(S_0\iota_{X^{\de}}S_3)+\Omega_2(S_3\iota_{X^{\de}}S_0) &= (S_1,\iota_{X^{\de}}S_3), \\
(S_2,\iota_{X^{\de}}S_0) + (S_0,\iota_{X^{\de}}S_2) &= \iota_{X^{\de}}\dr S_3, \\
(S_2,\iota_{X^{\de}}S_2) &= \iota_{X^{\de}}\dr S_4 + \iota_{X^{\de}}\dr S_5.
\end{align}\end{lem}
As in the proof of the mQME, these are all the relevant brackets and products for our choice of degrees. Since $S_3$ and $S_4$ are the only terms with differential forms that are not closed, all the terms we need are produced and we conclude the statement. \end{proof}
\section{Conclusions and outlook}
We have shown that the BV-BFV formalism can be applied to split Chern-Simons theory and produces a non-trivial example. Also, it seems plausible that using the method applied in section \ref{Integration} it is possible to make statements about the effective action to all orders. Furthermore, the structure of the identities in lemmata \ref{lemma1} and \ref{lemma2} makes it plausible that the structure of the effective action should in fact be governed by the mQME alone, i.e. that one can recover the state in the perturbed theory from the state in the unperturbed theory requiring only that the mQME is satisfied. A natural question to consider would be: to what extent one can make  such a statement rigorous, and in what generality one can prove it. \\
In another direction, in a next step we will use the state on the solid torus to compute the Chern-Simons theta invariants of lens spaces via the gluing operation. The relatively simple expression for the effective action in terms of a propagator and the cohomology on the boundary also should allow for an extension to higher genus handlebodies and other background flat connections, and thereby the computation of the Chern-Simons invariants for all 3-manifolds. 
\bibliographystyle{halpha}
\bibliography{bibliographyVdL}
\end{document}